\documentclass{article}
\usepackage{amsmath,amssymb,amsthm,graphicx,comment}
\usepackage[T2A]{fontenc}
\usepackage[english]{babel}
\usepackage[left=2.5cm,right=2.2cm,top=2.5cm,bottom=3.5cm]{geometry}
\usepackage{forest}
\usetikzlibrary{arrows.meta}
\usepackage[upright]{fourier}
\usepackage{tkz-graph}
\usetikzlibrary{shapes.geometric, arrows}
\usetikzlibrary{automata,positioning}

\let\le\leqslant
\let\ge\geqslant
\let\leq\leqslant
\let\geq\geqslant

\newtheorem{theorem}{Theorem}

\newtheorem{lemma}{Lemma}
\newtheorem*{remark*}{Remark}
\newtheorem{statement}{Statement}

\theoremstyle{definition}
\newtheorem{definition}{Definition}
\newtheorem*{hypothesis*}{Hypothesis}
\title{Maximum spread of vertex degrees in a simple graph}
\author{Sergey Onishchenko\footnote{Saint Petersburg University, Faculty of Mathematics and Computer Science, 7/9 Universitetskaya nab., St. Petersburg, 199034 Russia. E-mail: dimer03@mail.ru}}
\date{April 9, 2026}

\begin{document}

\maketitle

\begin{abstract}
We consider the following problem: let $n>k$ --- positive integers, $G$ ---
a graph on $n$ vertices (undirected, without loops or multiple edges). 
Let $h_k(G)$ denote the number of unordered pairs of vertices of the graph $G$ whose degrees differ by less than $k$. We seek to determine the smallest
possible value $f(n,k)$ of $h_k(G)$. 
The interest in this question is motivated by the fact that the bipartite analogue of the problem allowed S. Cichomski and F. Petrov \cite{CP} to prove the Burdzy--Pitman conjecture on the spread of independent identically distributed random variables. 
\end{abstract}

Keywords: vertex degree, simple graphs

\section{Introduction}
We consider the following problem: let $n>k$ --- positive integers, $G$ ---
a graph on $n$ vertices (undirected, without loops or multiple edges). 
Let $h_k(G)$ denote the number of unordered pairs of vertices of the graph $G$ whose degrees differ by less than $k$. We seek to determine the smallest
possible value $f(n,k)$ of $h_k(G)$. 
The interest in this question is motivated by the fact that the bipartite analogue of the problem allowed S. Cichomski and F. Petrov \cite{CP} to prove the Burdzy--Pitman conjecture on the spread of independent identically distributed random variables. 

The problem was first posed in our paper \cite{CP}. There, a conjectured answer to the problem was given, an upper bound was proved, and a lower bound was proved under several restrictions on $n$ and $k$. In this paper, a complete proof of the conjecture is presented.

\newpage


\section{\Large Statement of the main result}

The following theorem gives an upper bound for $f(n,k)$. 
\begin{theorem}\label{t1}
There exists a graph 
$G_{n,k}$ on $n$ vertices such that $h_k(G_{n,k})=f_0(n,k)$, where
$$
f_0(n,k):= 
\left(\left\lceil \frac{n}{k} \right\rceil - 2\right){k\choose 2} + {k+1\choose 2} + 
{n-k \left( \left\lceil\frac{n}{k} \right\rceil - 1 \right)-1
\choose 2}
$$
Consequently, $f(n,k)\leqslant f_0(n,k)$.
\end{theorem}

The following theorem gives a lower bound for $f(n,k)$. 
\begin{theorem}\label{t2}
$h_k(G)\geqslant f_0(n,k)$ for every graph
$G$ on $n$ vertices. 
\end{theorem}


\section{\Large Upper bound}

\begin{definition}\label{main} Let $G=(V,E)$ be a graph with $n$ vertices, $k<n$.
Define the numbers $i,t$ as follows: 
$n=ki+t$, $1 \leq t \leq k$, i.e. $i=\lceil n/k\rceil-1$, $t=n-k(\lceil n/k\rceil-1)$.
\end{definition}

\begin{definition}\label{main}
We call the unordered pair of vertices $(x,y)$ of $G$ close if $|\deg(x)-\deg(y)| < k$.
Let $A(G)$ denote the set of close pairs:
$A(G)=\{(x,y)|x,y \in V(G), |\deg(x)-\deg(y)| < k\}$.
\end{definition}

\begin{proof}[Proof of Theorem \ref{t1}]
Construct such a graph. Partition the vertices into $i+1$ groups 
$V_0,\ldots,V_i$ numbered 
$0, \dots, i$. Group $V_0$ contains $t-1=n-ki-1$ vertices, and group $V_{\lfloor (i+1)/2\rfloor}$ contains $k+1$
vertices, while the others contain $k$ vertices each (the total number of vertices is
$n-ki-1+ki+1=n$, as required).)

Edges between two vertices (which may lie in 
the same group or in different groups) are added
if and only if the sum of their group indices is strictly greater than $i$. In particular,
vertices in group $V_0$ have degree 0.
Let us verify that this graph works. Take a vertex $v$ from group $V_j$,
$j\in \{1,\ldots,i\}$, and compute its
degree. Vertex $v$ is adjacent to all vertices in the $j$ groups 
$V_{i-j+1},\ldots,V_i$ and
and to no others. Then into all these groups, except possibly the groups with indices $j$ and
$\lfloor (i+1)/2\rfloor$
there are exactly $k$ edges (there is one fewer edge to group $V_j$ because there is no edge
to vertex $v$ itself; there is one extra edge to group $V_{\lfloor (i+1)/2\rfloor}$
because it contains one extra vertex; if they are the same group,
these effects cancel and the number of edges to group $V_j$ is $k$)
Notice that among the groups $V_{i-j+1},\ldots,V_i$ the group
$V_j$ if and only if 
$j\geqslant i-j+1 \Leftrightarrow j\geqslant \lceil (i+1)/2\rceil$,
and the group $V_{\lfloor (i+1)/2\rfloor}$ --- if and only if
$i-j+1\leqslant \lfloor (i+1)/2\rfloor \Leftrightarrow j\geqslant (i+1)-
\lfloor (i+1)/2\rfloor=\lceil (i+1)/2\rceil$ --- so either both of these special groups are present or neither is. Thus, 
in total, from our vertex
there are always exactly $kj$ edges to the groups $V_{i-j+1},\ldots,V_i$. Therefore, the pairs of vertices of $G$
whose degrees differ by less than $k$ are exactly the pairs
of vertices from the same group. Obviously, there are exactly $f_0(n,k)$ such pairs.
\end{proof}

\newpage


\section{\Large Lower bound - proof of Theorem \ref{t2}}

\subsection{Inequalities used below}

\begin{statement}\label{l5}
    Let $i\geqslant 2$ be a positive integer,
    $a_1,\ldots,a_i$ and 
    $b_0,\ldots,b_{i-1}$ be nonnegative numbers. 
    Also set 
    $a_0=b_i:=0$.
    Then $$\sum_{j=0}^{i-1} b_ja_{j+1}\geqslant \min\{(b_j+a_j)
    (b_{j+1}+a_{j+1}),
    0\le j\le i-1\}.$$
\end{statement}

\begin{proof}
    Induction on $i$. 
    
    Base case $i=2$: $b_0a_1+b_1a_2\geqslant \min(b_0,a_2)(a_1+b_1) = \min\{b_0(a_1+b_1), (b_1+a_1)a_2\}$.

    Inductive step from $i-1$ to $i\ge3$. Using the induction hypothesis,
    it suffices to verify the inequality 
    \begin{align*}
        &\min\{b_0(b_1+a_1),\ldots,
    (b_{i-2}+a_{i-2})a_{i-1}\}+b_{i-1}a_i \ge \\
    &\min\{b_0(b_1+a_1),\ldots,
    (b_{i-1}+a_{i-1})a_{i}\}
    \end{align*}
    \\
    Either the minimum on the left-hand side is no smaller than that on the right-hand side (and then the inequality holds), or the left-hand side equals  
    \begin{align*}
        (b_{i-2}+a_{i-2})\cdot a_{i-1} + b_{i-1}\cdot a_i\geqslant (b_{i-1}+a_{i-1})\min\{a_i,b_{i-2}+a_{i-2}\}\\
        = \min\{(b_{i-1}+a_{i-1})a_i, (b_{i-1}+a_{i-1})(b_{i-2}+a_{i-2})\} \geq \min\{b_0(a_1+b_1), \dots ,(b_{i-1}+a_{i-1})a_i\}
    \end{align*}
\end{proof}

\begin{statement}\label{l5!}
    Let $i\geqslant 2$ be a positive integer,
    $a_0,\ldots,a_i$, $b_0,\ldots,b_i$, and 
    $c_0,\ldots,c_{i-1}$ be nonnegative numbers.
    Also set 
    $c_i:=0$.
    Then either $$\sum_{j=0}^{i-1} c_ja_{j+1}+c_jb_{j+1}+b_ja_{j+1}\geqslant \min\{(a_j+b_j+c_j)
    (a_{j+1}+b_{j+1}+c_{j+1}),
    1\le j\le i-1\},$$ or $$\sum_{j=0}^{i-1} c_ja_{j+1}+c_jb_{j+1}+b_ja_{j+1} \geqslant c_0(a_1+b_1+c_1) +$$ $$+ \min\{(a_1+b_1+c_1)(a_2+b_2+c_2), \dots , (a_{i-2}+b_{i-2}+c_{i-2})(a_{i-1}+b_{i-1}+c_{i-1}) , (a_{i-1}+b_{i-1}+c_{i-1})a_i\}$$ or $$\sum_{j=0}^{i-1} c_ja_{j+1}+c_jb_{j+1}+b_ja_{j+1} \geqslant (b_0+c_0)(a_1+b_1+c_1)$$
\end{statement}
\begin{proof} 1) Suppose $b_0+c_0 \leq c_0+a_2, a_2+b_2$, then $(b_0+c_0)(a_1+b_1+c_1) = b_0a_1+c_0(a_1+b_1)+(b_0+c_0)c_1+b_0b_1 \leq b_0a_1+c_0(a_1+b_1)+(a_2+b_2)c_1+a_2b_1$. Therefore \\\
    $\sum_{j=0}^{i-1} c_ja_{j+1}+c_jb_{j+1}+b_ja_{j+1} \geqslant \sum_{j=0}^{i-1} (c_ja_{j+1}+c_jb_{j+1}+b_ja_{j+1}) + (b_0+c_0)(a_1+b_1+c_1) - b_1a_2 - c_1a_2 - c_1b_2 - b_0a_1 - c_0(a_1+b_1) \geqslant (b_0+c_0)(a_1+b_1+c_1)$

2) Suppose $c_0+a_2 \leq b_0+c_0, a_2+b_2$, then $(c_0+a_2)(a_1+b_1+c_1) = (c_1+b_1)a_2 + c_0(a_1+b_1) + a_2a_1 + c_1c_0 \leq (c_1+b_1)a_2 + c_0(a_1+b_1) + b_0a_1 + c_1b_2$. Therefore \\\
    $\sum_{j=0}^{i-1} c_ja_{j+1}+c_jb_{j+1}+b_ja_{j+1} \geqslant \sum_{j=0}^{i-1} (c_ja_{j+1}+c_jb_{j+1}+b_ja_{j+1}) + (c_0+a_2)(a_1+b_1+c_1) - (c_1+b_1)a_2 - c_0(a_1+b_1) - b_0a_1 - c_1b_2 \geqslant c_0(a_1+b_1+c_1) + (a_1+b_1+c_1)a_2 + \sum_{j=2}^{i-1} (c_j+b_j)a_{j+1} \geqslant c_0(a_1+b_1+c_1) + \min\{(a_1+b_1+c_1)(a_2+b_2+c_2), \dots , (a_{i-2}+b_{i-2}+c_{i-2})(a_{i-1}+b_{i-1}+c_{i-1}) , (a_{i-1}+b_{i-1}+c_{i-1})a_i\}$ (the last inequality follows from Statement \ref{l5}, applied to $a_j$ and $b_j+c_j$)
    
3) Suppose $a_2+b_2 \leq c_0+a_2, b_0+c_0$, then $(b_2+a_2)(a_1+b_1+c_1) = (a_2+b_2)c_1 + (a_2+b_2)a_1 +(a_2+b_2)b_1 \leq (a_2+b_2)c_1 + (b_0+c_0)a_1 +(c_0+a_2)b_1$. Therefore\\
    $\sum_{j=0}^{i-1} c_ja_{j+1}+c_jb_{j+1}+b_ja_{j+1} \geqslant \sum_{j=0}^{i-1} (c_ja_{j+1}+c_jb_{j+1}+b_ja_{j+1}) + (b_2+a_2)(a_1+b_1+c_1) - (a_2+b_2)c_1 - (b_0+c_0)a_1 - (c_0+a_2)b_1 \geqslant (b_2+a_2)(a_1+b_1)\sum_{j=1}^{i-1}c_j(a_{j+1}+b_{j+1}) \geqslant \min\{(a_j+b_j+c_j)
    (a_{j+1}+b_{j+1}+c_{j+1}),
    1\le j\le i-1\}$ (the last inequality follows from Statement \ref{l5}, applied to $a_j+b_j$ and $c_j$)
\end{proof}

\begin{statement}\label{l35}
Let $m$ be a positive integer,
$x_1,\ldots,x_m$ be integers,
$\theta(x)$ be a function of an integer argument,
defined on the interval $[\min(x_1,\ldots,x_m),\max(x_1,\ldots,x_m)]$ and convex on it, that is,
the difference
$\theta(x)-\theta(x-1)$ is increasing.
Suppose that the integers $A,B,k$
satisfy $A+B=m$, $Ak+B(k+1)=\sum x_i$. Then $\sum \theta(x_i)\ge A\cdot \theta(k)+B\cdot \theta(k+1)$.
\end{statement}

\begin{proof}   Consider two cases.

1) Both numbers $A,B$ are nonnegative. 
Then, by shifting the integers $x_1,\ldots,x_m$ while preserving
    the sum, we can make any two of them differ by at most 1
    (if some two differ by more than 1, we move them closer while preserving the sum. The process terminates because the sum of the numbers remains unchanged while the sum of their squares strictly decreases).
    The left-hand side of our inequality does not increase in the process.
    Eventually any two numbers differ by at most 1, and then $A$ of them are equal to $k$,
    while the remaining $B$ are equal to $k+1$ --- hence the left-hand side becomes equal to the right-hand side. Since the left-hand side did not increase during the process, it was initially no smaller than the right-hand side.

    2) One of the numbers $A,B$ is negative (for example, $B$; the other case is analogous). Then rewrite the inequality as
    $\sum \theta(x_i)+(-B)\theta(k+1)\geqslant (m-B)\theta(k)$
    and apply the result proved in part 1) to the numbers
    $x_1,\ldots,x_m$ together with $(-B)$ copies of $k+1$, of which there are $m-B$ in total, whose sum is $(m-B)k$.
\end{proof}

\subsection{Beginning of the proof}
First consider the case when $t$ is small.

\begin{lemma}\label{t0}
The theorem holds when 
$t=1$.
\end{lemma}
\begin{proof}
Suppose the theorem is false, with graph $G$ as a counterexample.
Without loss of generality, $G$ has no vertices of degree $n-1$ (otherwise pass to the complement graph).

Partition the numbers from 0 to $n-1$ into $i+1$ intervals
as follows: the first $i$ intervals contain $k$ numbers each,
the last one contains $1$ number. Accordingly,
the vertex set of $G$ is partitioned into $i+1$ groups. One of them is empty. Hence, the number of close pairs of vertices in $G$ is at least $\sum_{j=0}^{i-1} {x_j\choose 2}, \sum x_j=n=ki+1$.

$\sum_{j=0}^{i} {x_j\choose 2} \geq {k+1\choose 2}+(i-1){k\choose 2} = f_0(n,k)$ - by Statement \ref{l35}, since ${x\choose 2}$ is a convex function.

Contradiction.
\end{proof}
From now on assume $t\geq2$

We now prove the lemma:

\begin{lemma}\label{l2}
Let graph $G$ be a counterexample to the theorem for some $n>k$,
$n=ki+t$, $1\leqslant t\leqslant k$.
Partition the set of numbers $\{0, 1, \dots, n-1\}$ into $i+1$ intervals (numbered from 0 to $i$) 
such that each interval has size at most $k$. Then for each interval
the graph $G$ contains at least $t$ vertices whose degrees lie in this interval.
\end{lemma}
\begin{proof}
Let $x_j$ be the number of vertices of $G$ with degree
in the $j$-th interval ($j=0,\ldots,i$). Then $h_k(G)\geqslant \sum {x_j\choose 2}$,
whereas $f_0(n,k)={t-1\choose 2}+{k+1\choose 2}+(i-1){k\choose 2}=\sum {y_j\choose 2}$,
where $y_0=t-1$, $y_1=k+1$, $y_2=\ldots=y_i=k$.
If the statement of the lemma is false, then $\min_j x_j\leqslant t-1$. By moving
the maximum and minimum among $x_1,\ldots,x_i$ closer while preserving
the sum, we can make one of them equal to $t-1$, while the sum
$\sum {x_j\choose 2}$ can only decrease. But when
one of the $x_j$ equals $t-1$, the sum of the others equals
$(k+1)+(i-1)k$, therefore, by convexity, the sum of the values of the convex function ${x\choose 2}$
over the remaining variables is at least ${k+1\choose 2}+(i-1){k\choose 2}$
(this can also be proved, for example, by moving the values closer while preserving the sum). Hence
$\sum_j {x_j\choose 2}\geqslant f_0(n,k)$ and $G$ is not a counterexample, a contradiction. 
\end{proof}

We will use the following term:

\begin{definition}\label{main} Let $G=(V,E)$ be a graph with $n$ vertices, $k<n$.
A subset of vertices of $G$ such that every pair of vertices in the subset is close will be called a vertex group of $G$.
\end{definition}

Lemma \ref{l2} gives us several partitions of the vertices into $i+1$ groups of size at least $t$ (each partition of $\{0, ... , n-1\}$ into intervals of length at most $k$ induces a partition of the vertices into groups). We will estimate the number of close pairs of vertices separately for pairs within one group and pairs from different groups, while choosing a suitable partition into groups. We divide the argument into 2 stages. In the first stage we do this roughly, in order to exclude simple cases and thereby reduce the computations needed in the second stage, which completes the proof.

\subsection{Stage 1 --- rough estimate}

Partition the numbers from 0 to $n-1$ into $i+1$ intervals containing at most $k$ numbers each. Accordingly,
the vertex set of $G$ is partitioned into $i+1$ groups.

\begin{lemma}\label{l3'}
Suppose that under every partition into groups, each group contains at least $t+p_2$ vertices for some integer $p_2\geq0$ (we can define $p_2$ so that this holds, according to Lemma \ref{l2}).
Let $x$ be the average degree of the vertices with degrees in $[0,t-1]$. Then the average degree of the vertices with degrees in $[n-t,n-1]$ is at most $ki+x-p_2-1$.
\end{lemma}
\begin{proof}
If the claim is false, then the average degree of the vertices with degrees in $[n-t,n-1]$ is greater than $ki+x-p_2-1$,
then the average non-degree of the vertices with degrees in $[n-t,n-1]$ is less than  $t-x+p_2$. Let 
$y,z$ --- be the numbers of vertices with degrees in $[0,t-1]$ and in $[n-t,n-1]$, respectively.
Then there are at most $xy$
edges and fewer than $(t-x+p_2)z$ non-edges between these groups. Hence 
    $yz < xy + (t-x+p_2)z\le (t+p_2)\max(y,z)$, which implies $\min(y,z)< t+p_2$ ---
    a contradiction.
\end{proof}

In the following lemma we estimate the number of close pairs of vertices within one group.

\begin{lemma}\label{l336}
    Suppose that the $n=ki+t$ vertices are partitioned into $i+1$ groups, and the size of the smallest group
    is $t+p_2$, while the smallest of the remaining groups has size $t+p_1$, where $p_1,p_2\geq0$. 
    If $p_1+p_2\leq k-t$, then the number of pairs of vertices belonging to the same
    group is at least
$$
f_0(n,k) - (-k^2/2+k(p_1+p_2)+kt+k/2-p_1^2/2-p_1t+p_1/2-p_2^2/2-p_2t+p_2/2-t^2/2-t/2+1).
$$
    If $p_1+p_2\geq k-t$, then the number of pairs of vertices belonging to the same
    group is at least
$$
f_0(n,k) - (-k^2/2+k(p_1+p_2)+kt+3k/2-p_1^2/2-p_1t-p_1/2-p_2^2/2-p_2t-p_2/2-t^2/2-3t/2+1).
$$
\end{lemma}

\begin{proof} We apply 
Statement \ref{l35} to the function ${x\choose 2}$; the numbers
$x_j$, $j=1,\ldots,i-1$, are the sizes of all groups except the two
smallest ones.
If $p_1+p_2\leq k-t$, we obtain 
$$
\sum {x_j\choose 2}\geqslant (i-1-k+t+p_1+p_2){k\choose 2}+(k-t-p_1-p_2){k+1\choose 2}.
$$
 If $p_1+p_2\geq k-t$, then
$$
\sum {x_j\choose 2}\geqslant (i-1+k-t-p_1-p_2){k\choose 2}+(-k+t+p_1+p_2){k-1\choose 2}.
$$

To complete the proof of the lemma it is enough to add
to these expressions
${t+p_1\choose 2}+{t+p_2\choose 2}$ and
and expand the brackets.
\end{proof}

Next we estimate the number of close pairs of vertices from different groups.

Now suppose that for every such partition the smallest group has at least $t+p_2$ vertices, the second smallest has at least $t+p_1$ vertices, with $p_1\geq p_2\geq0$. 

Let $x$ --- the average degree of the vertices with degrees in $[0,t-1]$. Let $y\leq ki+x-p_2-1$ be the average degree of the vertices with degrees in $[n-t,n-1]$ (we used Lemma \ref{l3'}). 
Since $y\geq n-t=ki$,
we obtain $x\geq p_2+1$.

For all $j$, $j=0,1,\ldots,i$, let $a_j$ be the number of vertices with degrees in $[kj,\frac{x+y-ki}{2}]$, $b_j$ the number of vertices with degrees in $[\frac{x+y-ki}{2},kj+t-1]$, and $c_j$ the number of vertices with degrees in $[kj+t,k(j+1)-1]$.

Also note that 
the desired number of close pairs of vertices from
different groups is at least $b_0a_1+c_0(a_1+b_1)+b_1a_2+c_1(a_2+b_2)+\ldots $. We estimate this using Statement \ref{l5!}.

Suppose the first case in Statement \ref{l5!} applies (as there, define $c_i=0$): $b_0a_1+c_0(a_1+b_1)+b_1a_2+c_1(a_2+b_2)+\ldots \geq \min\{(a_j+b_j+c_j)
    (a_{j+1}+b_{j+1}+c_{j+1}),
    1\le j\le i-1\} \geq (t+p_1)(t+p_2)$.

Suppose $p_1+p_2\geq k-t$. According to Lemma \ref{l336}, to complete the proof it remains to verify the inequality: \\
$(t+p_1)(t+p_2) > (-k^2/2+k(p_1+p_2)+kt+3k/2-p_1^2/2-p_1t-p_1/2-p_2^2/2-p_2t-p_2/2-t^2/2-3t/2+1)$ \\
As a function of $k$, the right-hand side is a downward-opening parabola with vertex at $p_1+p_2+t+3/2$ - hence (taking the boundary condition into account) it is enough to check the inequality at $k=p_1+p_2+t$ - and under this substitution the inequality becomes: \\
$(t+p_1)(t+p_2) > p_1p_2+p_1+p_2+1$ - which follows from Lemma \ref{t0}.

Suppose $p_1+p_2\leq k-t$. According to Lemma \ref{l336}, to complete the proof it remains to verify the inequality: \\
$(t+p_1)(t+p_2) > (-k^2/2+k(p_1+p_2)+kt+k/2-p_1^2/2-p_1t+p_1/2-p_2^2/2-p_2t+p_2/2-t^2/2-t/2+1)$ \\
As a function of $k$, the right-hand side is a downward-opening parabola with vertex at $p_1+p_2+t+1/2$ - hence it is enough to check the inequality at $k=p_1+p_2+t$ - and under this substitution the inequality becomes: \\
$(t+p_1)(t+p_2) > p_1p_2+p_1+p_2+1$ - which follows from Lemma \ref{t0}. \\

Suppose the second case in Statement \ref{l5!} applies: $b_0a_1+c_0(a_1+b_1)+b_1a_2+c_1(a_2+b_2)+\ldots \geq c_0(a_1+b_1+c_1) + \min\{(a_1+b_1+c_1)(a_2+b_2+c_2), \dots , (a_{i-2}+b_{i-2}+c_{i-2})(a_{i-1}+b_{i-1}+c_{i-1}) , (a_{i-1}+b_{i-1}+c_{i-1})a_i\}$. \\
If the last minimum equals $\min\{(a_1+b_1+c_1)(a_2+b_2+c_2), \dots , (a_{i-2}+b_{i-2}+c_{i-2})(a_{i-1}+b_{i-1}+c_{i-1})\}$, the theorem follows in the same way as in the first case.
\\

Therefore, at Stage 2 we may assume that, with the notation above, the number of close pairs from different groups is at least \\
$\min\{(b_0+c_0)(a_1+b_1+c_1), c_0(a_1+b_1+c_1) + (a_{i-1}+b_{i-1}+c_{i-1})a_i\}$ (this is the expression from the third case of Statement \ref{l5!} and the expression from the second case of Statement \ref{l5!}, where the large minimum is replaced by $(a_{i-1}+b_{i-1}+c_{i-1})a_i$)

\subsection{Stage 2 --- completing the proof}

First we refine Lemma \ref{l3'}.

\begin{lemma}\label{l3}
Denote the two numbers --- the number of vertices with degrees in $[0,t-1]$ and the number of vertices with degrees in $[n-t,n-1]$ --- by $t+p_2$ and $t+p_3$, respectively, where $p_3\geq p_2$.
Let $x$ be the average degree of the vertices with degrees in $[0,t-1]$. Then the average degree of the vertices with degrees in $[n-t,n-1]$ is at most $ki+x-\frac{(t+p_3)(p_2+1)}{t+p_2}$.
\end{lemma}
\begin{proof}
Let $z$ - the average non-degree of the vertices with degrees in $[n-t,n-1]$. We need to prove that $x+z \geq \frac{(t+p_3)(p_2+1)}{t+p_2}+t-1$. Suppose not, $x+z < \frac{(t+p_3)(p_2+1)}{t+p_2}+t-1$. Without loss of generality, assume that the number of vertices with degrees in $[0,t-1]$ is $t+p_2$ (otherwise pass to the complement graph). Then the number of edges between the groups under consideration is\\
$(t+p_2)(t+p_3) \leq (t+p_2)x+(t+p_3)z \leq (t+p_2)(x+z-t+1) + (t+p_3)(t-1) < \\ \\ < (t+p_2) \frac{(t+p_3)(p_2+1)}{t+p_2} + (t+p_3)(t-1) = (t+p_2)(t+p_3)$ --- a contradiction.
\end{proof}

Now refine Lemma \ref{l336} to estimate the number of close pairs of vertices within one group.

\begin{lemma}\label{l36}
    Let $p_1, p_2, p_3 \in \mathbb{Z}_{\geq 0}$, $y_0, ... , y_i \in \mathbb{Z}_{\geq 0}$, with three of these numbers equal to $t+p_1$,
     $t+p_2$ and $t+p_3$, respectively, and $\sum y_j = n$. 
    If $p_1+p_2+p_3\leq2(k-t)$, then
\begin{align*}
\sum {y_j \choose 2} \geq (i-2-2k+2t+p_1+p_2+p_3){k\choose 2}+(2k-2t-p_1-p_2-p_3){k+1\choose 2} = \\
= f_0(n,k) - (-k^2+2kt-t^2+1-\frac{p_1^2}{2}+p_1k-p_1t+\frac{p_1}{2}-\frac{p_2^2}{2}+ \\
+p_2k-p_2t+\frac{p_2}{2}-\frac{p_3^2}{2}+p_3k-p_3t+\frac{p_3}{2}).
\end{align*}
    If $p_1+p_2+p_3\geq2(k-t)$, then 
\begin{align*}
\sum {y_j \choose 2} \geq (i-2+2k-2t-p_1-p_2-p_3){k\choose 2}+(-2k+2t+p_1+p_2+p_3){k-1\choose 2} = \\
=f_0(n,k) - (-k^2+2kt+2k-t^2-2t+1-\frac{p_1^2}{2}+p_1k-p_1t-\frac{p_1}{2}-\frac{p_2^2}{2}+ \\
+p_2k-p_2t-\frac{p_2}{2}-\frac{p_3^2}{2}+p_3k-p_3t-\frac{p_3}{2}).
\end{align*}

Denote this expression by $f'(p_1,p_2,p_3)$
\end{lemma}

\begin{proof} We apply 
Statement \ref{l35} to the function ${x\choose 2}$, where $x_j$ are all $y_j$ from the statement of the lemma except those equal to $t+p_1$, $t+p_2$, and $t+p_3$.
If $p_1+p_2+p_3 \leq 2(k-t)$, we obtain 
$$
\sum {x_j\choose 2}\geqslant (i-2-2k+2t+p_1+p_2+p_3){k\choose 2}+(2k-2t-p_1-p_2-p_3){k+1\choose 2}.
$$
 If $p_1+p_2+p_3 \geq 2(k-t)$, then
$$
\sum {x_j\choose 2}\geqslant (i-2+2k-2t-p_1-p_2-p_3){k\choose 2}+(-2k+2t+p_1+p_2+p_3){k-1\choose 2}.
$$

To complete the proof of the lemma it is enough to add
to these expressions
${t+p_1\choose 2}+{t+p_2\choose 2}+{t+p_3\choose 2}$ and
and expand the brackets.
\end{proof}

We now estimate the number of close pairs of vertices from different groups to obtain an overall bound.

\begin{lemma}\label{l4}
Define $p_3\geq p_2$ as in Lemma \ref{l3}.

Then $p_1\geq0$ can be chosen so that in such a graph $G$ there are at least $f'(p_1,p_2,p_3)+\frac{(t+p_1)(t+p_2)(t+p_3)(p_2+1)}{2t^2+tp_2-3t-2p_2-p_2p_3-p_3}$ close pairs of vertices, where $2t^2+tp_2-3t-2p_2-p_2p_3-p_3 > 0$.
\end{lemma}
\begin{proof}
As before, let $x$ be the average degree of the vertices with degrees in $[0,t-1]$. Let $y\leq ki+x-\frac{(t+p_3)(p_2+1)}{t+p_2}$ be the average degree of the vertices with degrees in $[n-t,n-1]$ (we used Lemma \ref{l3}). 
Since all these vertices have degree at least $n-t=ki$,
obtain $x\geq \frac{(t+p_3)(p_2+1)}{t+p_2}$.

Partition the numbers from 0 to $n-1$ into $i+1$ intervals
as follows: the first $i$ intervals contain $k$ numbers each,
the last one contains $t$ numbers. Accordingly,
the vertex set of $G$ is partitioned into $i+1$ groups.

In group $j$,
$j=0,\ldots,i$, there are exactly $a_j$ vertices
whose degrees lie in 
$[kj, \frac{x+y-ki}{2}+kj]$, , $b_j$ is the number of vertices with degrees $(\frac{x+y-ki}{2}, kj+t-1]$, $c_j$ is the number of vertices with degrees $[kj+t, k(j+1)-1]$.
As established in Stage 1, we may assume that the number of close pairs of vertices from different groups is at least $\min\{(b_0+c_0)(a_1+b_1+c_1), c_0(a_1+b_1+c_1) + (a_{i-1}+b_{i-1}+c_{i-1})a_i\}$ 

We estimate $b_0$ from below.
Let the fraction of vertices with degree greater than 
$\frac{x+y-ki}{2}$
(among the vertices with degrees in $[0,t-1]$) be $r$, and let 
$q=1-r$. Then the average degree 
of the zero-th group is at most 
\begin{align*}
x&\leq r \cdot (t-1)+q\cdot \left(\frac{x+y-ki}{2}
\right)\leqslant  r(t-1)+q \cdot \left(\frac{2x-\frac{(t+p_3)(p_2+1)}{t+p_2}}{2}\right)\\
&r \cdot (t-1) - \frac{(1-r)(t+p_3)(p_2+1)}{2(t+p_2)} \geqslant rx \geqslant \frac{r(t+p_3)(p_2+1)}{t+p_2}\\
&r \cdot \left (t-1-\frac{(t+p_3)(p_2+1)}{2(t+p_2)} \right) \geqslant \frac{(t+p_3)(p_2+1)}{2(t+p_2)}\\
&r \cdot (2t^2+2tp_2-2t-2p_2-tp_2-t-p_2p_3-p_3) \geqslant (t+p_3)(p_2+1),
\end{align*}
which implies $2t^2+2tp_2-2t-2p_2-tp_2-t-p_2p_3-p_3>0$ and $1\geqslant r\geqslant \frac{(t+p_3)(p_2+1)}{2t^2+2tp_2-2t-2p_2-tp_2-t-p_2p_3-p_3}$ (the left inequality follows from the definition of $r$).

Now estimate $a_i$ from below.
Let the fraction of vertices with degree 
$\leqslant \frac{x+y+ki}{2}$ 
(among the vertices with degrees in $[ki,ki+t-1]$) be $r'$, and let $q'=1-r'$. 
Then the average degree $y$ in the $i$-th group is at least 
\begin{align*}
   y&\geq  r' \cdot ki+q' \cdot \frac{x+y+ki}{2}
   \geq  r' \cdot ki+q' \cdot \left(\frac{2y+\frac{(t+p_3)(p_2+1)}{t+p_2}}{2}\right)\\
   &r' \cdot ki + \frac{(1-r')(t+p_3)(p_2+1)}{2(t+p_2)} \leqslant r'y \leqslant r' \left(ki+x-\frac{(t+p_3)(p_2+1)}{t+p_2} \right)\\
&\frac{(1-r')(t+p_3)(p_2+1)}{2(t+p_2)} \leqslant r' \left(x-\frac{(t+p_3)(p_2+1)}{t+p_2} \right) \leqslant r' \left(t-1-\frac{(t+p_3)(p_2+1)}{t+p_2} \right)\\
&\frac{(t+p_3)(p_2+1)}{2(t+p_2)} \leqslant r' \left(t-1-\frac{(t+p_3)(p_2+1)}{2(t+p_2)} \right)\\
&r' \cdot (2t^2+2tp_2-2t-2p_2-tp_2-t-p_2p_3-p_3) \geqslant (t+p_3)(p_2+1),
\end{align*}
which implies $2t^2+2tp_2-2t-2p_2-tp_2-t-p_2p_3-p_3>0$ and $1\geqslant r'\geqslant \frac{(t+p_3)(p_2+1)}{2t^2+2tp_2-2t-2p_2-tp_2-t-p_2p_3-p_3}$ (the left inequality follows from the definition of $r'$).
\\

Then \\
$\min\{(b_0+c_0)(a_1+b_1+c_1), c_0(a_1+b_1+c_1) + (a_{i-1}+b_{i-1}+c_{i-1})a_i\} = \\ = c_0(a_1+b_1+c_1) + \min\{b_0(a_1+b_1+c_1), (a_{i-1}+b_{i-1}+c_{i-1})a_i\} \geqslant \\
c_0(a_1+b_1+c_1) + \min\{\left(\frac{(t+p_3)(p_2+1)}{2t^2+2tp_2-2t-2p_2-tp_2-t-p_2p_3-p_3}(a_0+b_0)\right)(a_1+b_1+c_1), \\ 
(a_{i-1}+b_{i-1}+c_{i-1})\left((a_i+b_i)\frac{(t+p_3)(p_2+1)}{2t^2+2tp_2-2t-2p_2-tp_2-t-p_2p_3-p_3}\right)\}$
\\

$\sum {a_j+b_j+c_j\choose 2}$ - the number of pairs of vertices from the same group (summed over the whole graph). Thus the total number of close pairs of vertices is at least \\
$\sum {a_j+b_j+c_j\choose 2} + c_0(a_1+b_1+c_1) + \min\{\left(\frac{(t+p_3)(p_2+1)}{2t^2+2tp_2-2t-2p_2-tp_2-t-p_2p_3-p_3}(a_0+b_0)\right)(a_1+b_1+c_1), \\ 
(a_{i-1}+b_{i-1}+c_{i-1})\left((a_i+b_i)\frac{(t+p_3)(p_2+1)}{2t^2+2tp_2-2t-2p_2-tp_2-t-p_2p_3-p_3}\right)\} \geq \\ \\
\sum_{j=2}^{i+1} {a_j+b_j+c_j\choose 2} + {a_0+b_0\choose 2} + {c_0+a_1+b_1+c_1\choose 2} + \min\{\left(\frac{(t+p_3)(p_2+1)}{2t^2+2tp_2-2t-2p_2-tp_2-t-p_2p_3-p_3}(a_0+b_0)\right)(c_0+a_1+b_1+c_1), \\ 
(a_{i-1}+b_{i-1}+c_{i-1})\left((a_i+b_i)\frac{(t+p_3)(p_2+1)}{2t^2+2tp_2-2t-2p_2-tp_2-t-p_2p_3-p_3}\right)\}$ \\
--- in the last step we used the fact that $1\geqslant \frac{(t+p_3)(p_2+1)}{2t^2+2tp_2-2t-2p_2-tp_2-t-p_2p_3-p_3}$ --- this result was already obtained when estimating $r$. \\

Take $t+p_1=c_0+a_1+b_1+c_1$ or $t+p_1=a_{i-1}+b_{i-1}+c_{i-1}$, depending on which of these components attains the minimum ($p_1\geq0$ by Lemma \ref{l2}). Then by Lemma \ref{l36}\\
$\sum_{j=2}^{i+1} {a_j+b_j+c_j\choose 2} + {a_0+b_0\choose 2} + {c_0+a_1+b_1+c_1\choose 2} + \min\{\left(\frac{(t+p_3)(p_2+1)}{2t^2+2tp_2-2t-2p_2-tp_2-t-p_2p_3-p_3}(a_0+b_0)\right)(c_0+a_1+b_1+c_1), \\ 
(a_{i-1}+b_{i-1}+c_{i-1})\left((a_i+b_i)\frac{(t+p_3)(p_2+1)}{2t^2+2tp_2-2t-2p_2-tp_2-t-p_2p_3-p_3}\right)\} \geqslant \\ \\
\frac{(t+p_1)(t+p_2)(t+p_3)(p_2+1)}{2t^2+tp_2-3t-2p_2-p_2p_3-p_3} +f'(p_1+p_2+p_3)$
--- that is, the lemma is proved
\end{proof}

It follows from Lemma \ref{l4} that to complete the proof of the theorem it is enough to show that $f'(p_1,p_2,p_3)+\frac{(t+p_1)(t+p_2)(t+p_3)(p_2+1)}{2t^2+tp_2-3t-2p_2-p_2p_3-p_3} \geq f_0(n,k)$. By Lemma \ref{l36} this is equivalent to the following 2 inequalities:

1) In the case $p_1+p_2+p_3 \leq 2(k-t)$: \\
$\frac{(t+p_1)(t+p_2)(t+p_3)(p_2+1)}{2t^2+tp_2-3t-2p_2-p_2p_3-p_3} \geqslant -k^2+2kt-t^2+1-\frac{p_1^2}{2}+p_1k-p_1t+\frac{p_1}{2}-\frac{p_2^2}{2}+p_2k-p_2t+\frac{p_2}{2}-\frac{p_3^2}{2}+p_3k-p_3t+\frac{p_3}{2}$ \\
As a function of $k$, the right-hand side is a downward-opening parabola with vertex at $t+p_1/2+p_2/2+p_3/2$ - hence it is enough to check the inequality at $k=t+p_1/2+p_2/2+p_3/2$ - and under this substitution the inequality becomes: \\
$\frac{(t+p_1)(t+p_2)(t+p_3)(p_2+1)}{2t^2+tp_2-3t-2p_2-p_2p_3-p_3} \geqslant -p_1^2/4-p_2^2/4-p_3^2/4+p_1p_2/2+p_1p_3/2+p_2p_3/2+p_1/2+p_2/2+p_3/2+1$

2) In the case $p_1+p_2+p_3 \geq 2(k-t)$: \\
$\frac{(t+p_1)(t+p_2)(t+p_3)(p_2+1)}{2t^2+tp_2-3t-2p_2-p_2p_3-p_3} \geqslant -k^2+2kt+2k-t^2-2t+1-\frac{p_1^2}{2}+p_1k-p_1t-\frac{p_1}{2}-\frac{p_2^2}{2}+p_2k-p_2t-\frac{p_2}{2}-\frac{p_3^2}{2}+p_3k-p_3t-\frac{p_3}{2}$ \\
As a function of $k$, the right-hand side is a downward-opening parabola with vertex at $t+1+p_1/2+p_2/2+p_3/2$ - hence (taking the boundary condition into account) it is enough to check the inequality at $k=t+p_1/2+p_2/2+p_3/2$ - and under this substitution the inequality also becomes: \\
$\frac{(t+p_1)(t+p_2)(t+p_3)(p_2+1)}{2t^2+tp_2-3t-2p_2-p_2p_3-p_3} \geqslant -p_1^2/4-p_2^2/4-p_3^2/4+p_1p_2/2+p_1p_3/2+p_2p_3/2+p_1/2+p_2/2+p_3/2+1$ \\

Notice that the right-hand side is no greater than $\frac{p_1p_2+p_2p_3+p_1+p_2+p_3}{2}+1$ - replace it by this expression\\
It is also known from Lemma \ref{l4} that the denominator on the left-hand side is positive. Multiply by it and expand the brackets. We obtain: \\
$3p_1p_2^2p_3/2+p_1p_2^2t/2+p_1p_2^2+p_1p_2p_3t+2p_1p_2p_3+2p_1p_2t+p_1p_2+p_1p_3t+p_1p_3/2+3p_1t/2+p_2^2p_3^2/2+p_2^2p_3t/2+3p_2^2p_3/2+p_2^2t^2-p_2^2t/2+p_2^2+p_2p_3^2+2p_2p_3t+5p_2p_3/2+p_2t^3+p_2t/2+2p_2+p_3^2/2+3p_3t/2+p_3+t^3-2t^2+3t \geq 0$ \\ 
--- which holds because $p_2^2t^2-p_2^2t/2+p_2^2 \geq 0$ and $t^3-2t^2+3t \geq 0$ \\
Theorem \ref{t2} is proved.


\section{Acknowledgement}
The work of S. Onishchenko was performed at the Saint Petersburg Leonhard Euler International Mathematical Institute and supported by the Ministry of Science and Higher Education of the Russian Federation (agreement no. 075–15–2025–343).

\end{document}